\documentclass[final]{siamltex}
\usepackage{latexsym,bm,amssymb,amsfonts,amsbsy,amsmath,graphics,graphicx,color}
\usepackage{url}
\def\dd{\,\mathrm{d}}

\newtheorem{remark}{Remark}
\newtheorem{exa}{Example}

\def \RR {{\mathbb{R}}}

\def \pmatrix{ \left( \begin{array} }
\def \endpmatrix{ \end{array} \right) }

\def\d{\mathrm{d}}

\begin{document}

\title{A two step, fourth order, nearly-linear method with energy preserving properties\thanks{Work developed
within the project ``Numerical methods and software for
differential equations''.}}

\author{Luigi Brugnano\thanks{Dipartimento di Matematica
``U.\,Dini'', Universit\`a di Firenze, Italy ({\tt
luigi.brugnano@unifi.it}).} \and Felice
Iavernaro\thanks{Dipartimento di Matematica, Universit\`a di Bari,
Italy ({\tt felix@dm.uniba.it}).} \and Donato
Trigiante\thanks{Dipartimento di Energetica ``S.\,Stecco'', Universit\`a di
Firenze, Italy ({\tt trigiant@unifi.it}).}}

\maketitle

\begin{abstract}
We introduce a family of fourth order two-step methods that preserve
the energy function of canonical polynomial Hamiltonian systems.
Each method in the family may be viewed as a correction of a linear
two-step method, where the correction term is $O(h^5)$ ($h$ is the
stepsize of integration). The key tools the new methods are based
upon are the line integral associated with a conservative vector
field (such as the one defined by a Hamiltonian dynamical system)
and its discretization obtained by the aid of a quadrature formula.
Energy conservation is equivalent to the requirement that the
quadrature is exact, which turns out to be always the case in the
event that the Hamiltonian function is a polynomial and the degree
of precision of the quadrature formula is high enough. The
non-polynomial case is also discussed and a number of test problems
are finally presented in order to compare the behavior of the new
methods to the theoretical results.
\end{abstract}

\vspace{-1em}

\vspace{-1em}

\section{Introduction and Background}
We consider canonical Hamiltonian systems in the form
\begin{equation}\label{ham}
  \frac{dy}{dt}= J \nabla H(y), \qquad
  J=\pmatrix{cc}0&I_m\\-I_m&0\endpmatrix, \qquad y(t_0)=y_0\in\RR^{2m},\vspace{-.5em}
\end{equation}
where $H(y)$ is a smooth real-valued function. Our interest is in
researching numerical methods that
 provide approximations $y_n\simeq y(t_0+nh)$ to the true solution along which
 the energy is precisely conserved, namely
\begin{equation}
\label{energy-conservation} H(y_n)=H(y_0), \qquad \mbox{for all
stepsizes } h\le h_0.
\end{equation}

The study of energy-preserving methods form a branch of
\textit{geometrical numerical integration}, a research topic whose
main aim is preserving qualitative features of simulated
differential equations. In this context, symplectic methods have had
considerable attention due to their good long-time behavior as
compared to standard methods for ODEs \cite{Ru,Fe,LeRe}. A related
interesting approach based upon exponential/trigonometric fitting
may be found in \cite{IxVB04,VB06,Si08}. Unfortunately,
symplecticity cannot be fully combined with the energy preservation
property \cite{GM}, and this partly explains why the latter has been
absent from the scene for a long time.

Among the first examples of energy-preserving methods we mention
discrete gradient schemes \cite{Gon96,McL99} which are defined by
devising discrete analogs of the gradient function. The first
formulae in this class had order at most two but recently discrete
gradient methods of arbitrarily high order have been researched by
considering the simpler case of systems with one-degree of freedom
\cite{CieRat10,CieRat10a}.

Here, the key tool we wish to  exploit is the well-known line
integral associated with conservative vector fields, such us the one
defined at \eqref{ham}, as well as its  discrete version, the so
called \textit{discrete line integral}. Interestingly, the line
integral provides a means to check the energy conservation property,
namely
$$\begin{array}{rl}
H(y(t_1))-H(y_0) & = \displaystyle \int_{y_0 \rightarrow y(t_1)}
\hspace*{-.6cm} \nabla H(y) \d y =  h\int_0^1 y'(t_0+\tau h)^T
\nabla
H(y(t_0+\tau h ))  \d \tau \\[.35cm]  & = h\displaystyle \int_0^1 \nabla^T H(y(t_0+\tau h
)) J^T  \nabla H(y(t_0+\tau h )) \d \tau = 0,
\end{array}
$$
with $h=t_1-t_0$, that can be easily converted into a discrete
analog by considering a quadrature formula in place of the integral.

The discretization process requires to change the curve $y(t)$ in
the phase space $\RR^{2m}$  to a simpler curve $\sigma(t)$
(generally but not necessarily a polynomial), which is meant to
yield the approximation at time $t_1=t_0+h$, that is
$y(t_0+h)=\sigma(t_0+h)+O(h^{p+1})$, where $p$ is the order of the
resulting numerical method. In a certain sense, the problem of
numerically solving \eqref{ham} while preserving the Hamiltonian
function is translated into a quadrature problem.

For example,  consider the segment $\sigma(t_0+ch)=(1-c)y_0+cy_{1}$,
with $c\in[0,1]$, joining $y_0$ to an unknown point $y_1$ of the
phase space. The line integral of $\nabla H(y)$ evaluated along
$\sigma$ becomes
\begin{equation}
\label{iav-eq1} {H(y_{1})-H(y_0)} =  h(y_{1}-y_0)^T \int_0^1  \nabla
H((1-c)y_0+cy_{1})  \dd c.
\end{equation}
Now assume that $H(y)\equiv H(q,p)$ is a polynomial of degree $\nu$
in the generalized coordinates $q$ and in the momenta $p$. The
integrand in \eqref{iav-eq1} is a polynomial of degree $\nu-1$ in
$c$ and can be exactly solved by any quadrature formula with
abscissae $c_1<c_2<\cdots<c_k$ in $[0,1]$ and weights
$b_1,\dots,b_k$, having degree of precision $d \ge \nu-1$. We thus
obtain
$$
H(y_{1})-H(y_0) =h(y_{1}-y_0)^T {\sum_{i=1}^k b_i\nabla
H((1-c_i)y_0+c_iy_{1})}.
$$
To get the energy conservation property we impose that $y_1-y_0$ be
orthogonal to the above sum, and in particular we choose (for the
sake of generality we use $f(y)$ in place of $J\nabla H(y)$ to mean
that the resulting method also makes sense when applied to a general
ordinary differential equation $y'=f(y)$)
\begin{equation}
\label{iav-s-stage-trap}
y_{1}=\displaystyle y_0+h\sum_{i=1}^kb_if(Y_i), \qquad
Y_i=(1-c_i)y_0+c_iy_{1}, \quad i=1,\dots,k.
\end{equation}
Formula \eqref{iav-s-stage-trap}  defines a Runge--Kutta method with
Butcher tableau $\begin{array}{c|c} c & c b^T \\ \hline & b^T
\end{array}$, where $c$ and $b$ are the vectors of the abscissae and
weights, respectively. The stages $Y_i$ are called \textit{silent
stages} since their presence does not affect the degree of
nonlinearity of the system to be solved at each step of the
integration procedure: the only unknown is $y_1$ and consequently
\eqref{iav-s-stage-trap} defines a mono-implicit method.
Mono-implicit methods of Runge--Kutta type have been researched in
the past by several authors (see, for example,
\cite{Ca75,Bo,CaSi,BuChMu} for their use in the solution of initial
value problems).

Methods such as \eqref{iav-s-stage-trap} date back to 2007
\cite{IP1,IT3} and are called $k$-stage trapezoidal methods since on
the one hand the choice $k=2$, $c_1=0$, $c_2=1$ leads to the
trapezoidal method and on the other hand all other methods evidently
become the trapezoidal method when applied to linear problems.

Generalizations of \eqref{iav-s-stage-trap} to higher orders require
the use of a polynomial $\sigma$ of higher degree and are based upon
the same reasoning as the one discussed above. Up to now, such
extensions have taken the form of Runge--Kutta methods
\cite{BIT1,BIT2,BIT3}. It has been shown that choosing a proper
polynomial $\sigma$ of degree $s$ yields a Runge--Kutta method of
order $2s$ with $k\ge s$ stages. The peculiarity of such
energy-preserving formulae, called Hamiltonian Boundary Value
Methods (HBVMs), is that the associated Butcher matrix has rank $s$
rather than $k$, since $k-s$ stages may be cast as linear
combinations of the remaining ones, similarly to the stages $Y_i$ in
\eqref{iav-s-stage-trap}.\footnote{A documentation about HBVMs,
Matlab codes, and a complete set of references  is available at the
url \cite{BIT0}.} As a consequence,  the nonlinear system to be
solved at each step has dimension $2ms$ instead of $2mk$, which is
better visualized by recasting the method in block-BVM form
\cite{BIT1}.

In the case where $H(y)$ is not a polynomial, one can still get a
\textit{practical} energy conservation by choosing $k$ large enough
so that the quadrature formula approximates the corresponding
integral to within machine precision. Strictly speaking, taking the
limit as $k \rightarrow \infty$ leads to limit formulae where the
integrals come back into play in place of the sums. For example,
letting $k\rightarrow \infty$ in \eqref{iav-s-stage-trap} just means
that the integral in \eqref{iav-eq1} must not be discretized at all,
which would yield the \textit{Averaged Vector Field} method
$y_1=y_0+h\int_0^1 f((1-c)y_0+cy_{1}) \dd c$, (see \cite{M2AN,QMc}
for details).

In this paper we start an investigation that follows a different
route. Unlike the case with HBVMs, we want now to take advantage of
the previously computed approximations to extend the class
\eqref{iav-s-stage-trap} in such a way to increase the order of the
resulting methods, much as the class of linear multistep method may
be viewed as a generalization of (linear) one step methods. The
general question we want to address is whether there exist $k$-step
mono-implicit energy-preserving methods of order greater than two.
Clearly, the main motivation is to reduce the computational cost
associated with the implementation of HBVMs.

The purpose of the present paper is to give an affermative answer to
this issue in the case $k=2$. More specifically, the method
resulting from our analysis, summarized by formula \eqref{twostep},
may be thought of as a nearly linear two-step method in that it is
the sum of a fourth order linear two-step method, formula
\eqref{twostep-lin}, plus a nonlinear correction of higher order.

The paper is organized as follows. In Section \ref{def_methods} we
introduce the general formulation of the method, by which we mean
that the integrals are initially not discretized to maintain the
theory at a general level. In this section we also report a brief
description of the HBVM of order four, since its  properties will be
later exploited to deduce the order of the new method: this will be
the subject of Section \ref{analysis_sec}. Section \ref{discr_sec}
is devoted to the discretization of the integrals, which will
produce the final form of the methods making them ready for
implementation. A few test problems are presented in Section
\ref{test_sec} to confirm the theoretical results.

\section{Definition of the method}\label{def_methods}
Suppose that $y_1$ is an approximation to the true solution $y(t)$
at time $t_1=t_0+h$, where $h>0$ is the stepsize of integration.
More precisely, we assume that
\begin{itemize}
\item[($A_1$)] $y(t_1)=y_1+O(h^{p+1})$ with $p \ge 4$;
\item[($A_2$)] $H(y_1)=H(y_0)$, which means that $y_1$ lies on the very same
manifold $H(y)=H(y_0)$ as the continuous solution $y(t)$.
\end{itemize}
The two above assumptions are fulfilled if, for example, we compute
$y_1$ by means of a HBVM (or an $\infty$-HBVM \cite{BIT2}) of order
$p\ge 4$. The new approximation $y_2\simeq y(t_2) \equiv y(t_0+2h)$
is constructed as follows.

 Consider the quadratic polynomial
$\sigma(t_0+ 2 \tau h)$ that interpolates the set of data
$\{(t_0+jh, y_j)\}_{j=0,1,2}$. Expanded along the Newton basis
$\{P_j(\tau)\}$ defined on the nodes $\tau_0=0$,
$\tau_1=\frac{1}{2}$, $\tau_2=1$, the polynomial $\sigma$ takes the
form (for convenience we order the nodes as $\tau_0, \tau_2,
\tau_1$)
\begin{equation}
\label{sigma} \sigma(t_0 + 2 \tau h) = y_0+(y_2-y_0)\tau
+2(y_2-2y_1+y_0)\tau(\tau-1).
\end{equation}
As $\tau$ ranges in the interval $[0,1]$, the $2m$-length vector
$\gamma(\tau) \equiv \sigma(t_0+2\tau h)$ describes a curve in the
phase space $\RR^{2m}$. The line integral of the conservative vector
field $\nabla H(y)$ along the curve $\gamma$ will match the
variation of the energy function $H(y)$, that is
$$\begin{array}{rl}
H(y_2)-H(y_0) &= \displaystyle \int_{y_0 \rightarrow y_2}
\hspace*{-.6cm} \nabla H(y) \d y =  \int_0^1 \left[ \gamma'(\tau)
\right]^T  \nabla
H(\gamma(\tau)) \,  \d \tau \\[.5cm]
&  \hspace*{-1.5cm} \displaystyle =(y_2-y_0)^T \int_0^1 \nabla
H(\gamma(\tau)) \, \mathrm{d}\tau  + 2(y_2-2y_1+y_0)^T \int_0^1
(2\tau-1) \nabla H(\gamma(\tau)) \, \mathrm{d}\tau.
\end{array}
$$
The energy conservation condition $H(y_2)=H(y_0)$ yields the
following equation in the unknown $z\equiv y_2$
\begin{equation}
\label{nonlin} (z-y_0)^T \int_0^1 \nabla H(\gamma(\tau)) \,
\mathrm{d}\tau = -2(z-2y_1+y_0)^T \int_0^1 (2\tau-1) \nabla
H(\gamma(\tau)) \, \mathrm{d}\tau.
\end{equation}
The method we are interested in has the form $y_2=\Psi_h(y_0,y_1)$,
where $\Psi_h$ is  implicitly defined by the following nonlinear
equation in the unknown $z$:
\begin{equation}
\label{nonlin_sys}
\displaystyle z=y_0+ 2 h J a(z) + \frac{r(z)}{||a(z)||_2^2} a(z),
\qquad \mbox{with}~~ a(z)=\int_0^1  \nabla H (\gamma(\tau)) \, \d
\tau,
\end{equation}
where the residual $r(z)$ is defined as
\begin{equation}\label{residual}
r(z) \equiv -2(z-2y_1+y_0)^T \int_0^1 (2\tau-1) \nabla
H(\gamma(\tau)) \, \mathrm{d}\tau.
\end{equation}
A direct computation shows that any solution $z^\ast$ of
\eqref{nonlin_sys} also satisfies \eqref{nonlin}. In the next
section we will show that \eqref{nonlin_sys} admits a unique
solution $y_2\equiv z^\ast$ satisfying the order condition
$y_2=y(t_0+2h)+O(h^5)$. Such a result will be derived by regarding
\eqref{nonlin_sys} as a perturbation of the HBVM of order $4$ and,
in turn, by comparing the two associated numerical solutions. To
this end and to better explain the genesis of formula
\eqref{nonlin_sys} and the role of the integrals therein, a brief
introduction of the HBVM formula of order four is in order.

\subsection{HBVM of order four} \label{rem1} Suppose that  both $y_1$ and $y_2$
are unknown (so now  $y_1$ is no longer given a priori as indicated
by assumption ($A_1$)): let us call them $u_1$ and $u_2$
respectively. For \eqref{nonlin} to be satisfied, we can impose the
two orthogonality conditions
\begin{equation}
\label{orthbvm}\left\{
\begin{array}{l}
\displaystyle u_2 - y_0 = \eta_1 h J \int_0^1 \nabla H(\gamma(\tau))
\,
\mathrm{d}\tau, \\[.5cm]
\displaystyle  u_2-2u_1+y_0   = \eta_2 h J \int_0^1 (2\tau-1) \nabla
H(\gamma(\tau)) \, \mathrm{d}\tau,
\end{array}
\right.
\end{equation}
giving rise to a system of two block-equations (the curve
$\gamma(\tau)=\sigma(t_0+2\tau h)$ is as in \eqref{sigma} with $u_1$
and $u_2$ in place of $y_1$ and $y_2$). Setting the free constants
$\eta_1$ and $\eta_2$ equal to $2$ and $3$, respectively, confers
the highest possible order, namely $4$, on the resulting method:
$u_2=y(t_0+2h)+O(h^5)$ (see \cite{IT3} for details).\footnote{Since
we are integrating the problem on an interval $[t_0,t_2]$ of length
$2h$, we have scaled the constants $\eta_1$ and $\eta_2$ by a factor
two with respect to the values reported in \cite{IT3}.} Furthermore,
it may be shown that the internal stage $u_1$ satisfies the order
condition $u_1=y(t_0+h)+O(h^4)$.

Evidently, the implementation of \eqref{orthbvm} on a computer
cannot leave out of consideration the issue of solving the integrals
appearing in both equations. Two different situations may emerge:
\begin{itemize}
\item[(a)] the Hamiltonian function $H(y)$ is a polynomial of degree
$\nu$. In such a case, the two integrals in \eqref{orthbvm} are
exactly computed by a quadrature formula having degree of precision
$d \ge 2\nu-1$.
\item[(b)] $H(y)$ is not a polynomial, nor do the two integrands
admit a primitive function in closed form. Again, an appropriate
quadrature formula can be used to approximate the two integrals to
within machine precision, so that no substantial difference is
expected during the implementation process by replacing the
integrals by their discrete counterparts.
\end{itemize}

Case (a) gives rise to an infinite family of Runge-Kutta methods,
each depending on the specific choice (number and distribution) of
nodes the quadrature formula is based upon (see \cite{BIT2} for a
general introduction on HBVMs and \cite{BIT3} for their relation
with standard collocation methods). For example, choosing $k$ nodes
according to a Gauss distribution over the interval $[0,1]$ results
in a method that precisely conserves the energy if applied to
polynomial canonical Hamiltonian systems with $\nu \le k$ and that
becomes the classical $2$-stage Gauss collocation method when $k=2$.
On the other hand, choosing a Lobatto distribution yields a
Runge-Kutta method that preserves  polynomial Hamiltonian functions
of degree $\nu \le k-1$ and that becomes the Lobatto IIIA method of
order four when $k=2$.

The method resulting from case (b) are undistinguishable from the
original formulae \eqref{orthbvm} in that they are energy-preserving
up to machine precision when applied to any regular canonical
Hamiltonian system. Stated differently, \eqref{orthbvm} may be
viewed as the limit of the family of HBVMs of order four, as the
number of nodes tends to infinity. For this reason the limit
formulae \eqref{orthbvm} have been called $\infty$-HBVM of order $4$
(see \cite{BIT2}).

\begin{remark} In the present context,  $y_1$ being a known quantity, the unknown
$z$ in \eqref{nonlin} cannot in general satisfy, at the same time, both
orthogonality conditions in \eqref{orthbvm}. However, since $y_1$
may be thought of as an approximation of order four to the quantity
$u_1$ in \eqref{orthbvm}, should we only impose the first
orthogonality condition, namely
\begin{equation}
\label{orth1} z - y_0 = 2 h J a(z),
\end{equation}
we would expect the residual $r(z)$ (the right hand side of
\eqref{nonlin}) to be very small.\footnote{By exploiting the result
in Lemma~\ref{lem2} below, it is not difficult to show that actually
\eqref{orth1} implies $r(z)=O(h^5)$. This aspect is further
emphasized in the numerical test section (see Table \ref{tab2}).}
This suggests that a solution to \eqref{nonlin} that yields an
approximation of high order to $y(t_0+2h)$ may be obtained by
allowing a small deviation from orthogonality in \eqref{orth1}. This
is accomplished by setting $z - y_0 = 2 h J a(z) + \delta a(z)$, and
by tuning the perturbation parameter $\delta$ in such a way that
\eqref{nonlin} be satisfied: this evidently gives
$\delta=\frac{r(z)}{||a(z)||_2^2}$ and we arrive at
\eqref{nonlin_sys}.
\end{remark}

\section{Analysis of the method}
\label{analysis_sec}

Results on the existence and uniqueness of a solution of
\eqref{nonlin_sys} as well as on its order of accuracy will be
derived by first analyzing the simpler nonlinear system
\begin{equation}
\label{M}
\displaystyle z=y_0+ 2 h J a(z), \qquad \mbox{with}~~ a(z)=\int_0^1
\nabla H (\gamma(\tau)) \, \d \tau,
\end{equation}
obtained by neglecting the correction term
$\frac{r(z)}{||a(z)||_2^2} a(z)$. For $z\in \RR^{2m}$ we set (see
\eqref{sigma})
\begin{equation}
\label{gammaz} \gamma_z(\tau) = y_0+(z-y_0)\tau
+2(z-2y_1+y_0)\tau(\tau-1),
\end{equation}
and (see \eqref{M})
\begin{equation}
\label{Phi} \Phi(z)=y_0+ 2 h J a(z).
\end{equation}
In the following $||\cdot||$ will denote the $2$-norm.

\begin{lemma}
\label{lem1} There exist positive constants $\rho$ and $h_0$ such
that, for $h\le h_0$, system \eqref{M} admits a unique solution
$\hat z$ in the ball $B(y_0,\rho)$ of center $y_0$ and radius
$\rho$.
\end{lemma}
\begin{proof}
We show that constants $h_0,\rho>0$ exist such that the function
defined in \eqref{Phi} satisfies the following two conditions for $h
\le h_0$:
\begin{itemize}
\item[(a)]$\Phi(z)$ is a contraction on
$B(y_0,\rho)$, namely
$$
\forall z,w\in B(y_0,\rho),\quad ||\Phi(z)-\Phi(w)|| \le L ||z-w||,
\qquad \mbox{with} \quad L<1;
$$
\item[(b)] $||\Phi(y_0)-y_0||\le (1-L) \rho$.
\end{itemize}
The contraction mapping theorem can then be applied to obtain the
assertion.

Let $B(y_0,\rho)$ a ball centered at $y_0$ with radius $\rho$. We
can choose $h'_0$ and $\rho$ small enough that the image set
$\Omega=\{\gamma_z(\tau): \tau\in[0,1],~ z\in B(y_0,\rho),~h\le
h'_0\}$ is entirely contained in a ball $B(y_0,\rho')$ which, in
turn, is contained in the domain of $\nabla^2H(y)$.\footnote{Notice
that, by definition, the set $\Omega$ is an open simply connected
subset of $\RR^{2m}$ containing $B(y_0,\rho)$ while, from the
assumption ($A_1$), decreasing $h$ causes the point $y_1$ to
approach $y_0$.}  We set
$$
M_\rho = \max_{w\in B(y_0,\rho')} \left\|\nabla^2H(w) \right\|.
$$
 From \eqref{M} and \eqref{gammaz} we have
$$
\frac{\partial a(z)}{\partial z} = \int_0^1
\nabla^2H(\gamma_z(\tau)) \frac{\partial \gamma_z}{\partial z} \, \d
\tau = \int_0^1 \nabla^2H(\gamma_z(\tau)) \, \tau(2\tau-1) \, \d
\tau
$$
and hence
$$
\left\|\frac{\partial a(z)}{\partial z}\right\| \le M_\rho \int_0^1
\tau |2\tau-1| \, \d \tau =\frac{1}{4} M_\rho.
$$
Consequently (a) is satisfied by choosing
\begin{equation}
\label{Lcond}
L=\frac{h}{2}M_\rho
\end{equation}
 and $h_0<\min\{\frac{2}{M_\rho},h'_0\}$.
Concerning  (b), we observe that
$$
\Phi(y_0)-y_0 = 2hJa(y_0) =2h J \int_0^1 \nabla
H(y_0+4(y_0-y_1)\tau(\tau-1)) \, \d \tau,
$$
hence $||\Phi(y_0) -y_0|| = 2h ||a(y_0)||$ with $||a(y_0)||$ bounded
with respect to $h$. Since $L$ vanishes with $h$ (see
(\ref{Lcond})), we can always tune $h_0$ in such a way that
$2h||a(y_0)|| \le (1-L) \rho$.
\end{proof}

\begin{lemma}
\label{lem2} The solution $\hat z$ of \eqref{M} satisfies
$y(t+2h)-\hat z=O(h^5)$.
\end{lemma}
\begin{proof}
Under the assumption ($A_1$), \eqref{M} may be regarded as a
perturbation of system \eqref{orthbvm}, since  $y_1$ and $u_1$ are
$O(h^5)$ and $O(h^4)$ close to $y(t+h)$ respectively.\footnote{This
also implies that $u_1-y_1=O(h^4)$.} Since $u_2=y(t+2h)+O(h^5)$,  we
can estimate the accuracy of $\hat z$ as an approximation of
$y(t+2h)$ by evaluating its distance from $u_2$.

Let $\tilde \gamma(\tau)$ be the underlying quadratic curve
associated with the HBVM defined by \eqref{orthbvm}, namely
\begin{equation}
\label{tgamma} \tilde \gamma(\tau) \equiv y_0+(u_2-y_0)\tau
+2(u_2-2u_1+y_0)\tau(\tau-1).
\end{equation}
Considering that (see \eqref{gammaz})
$$
\gamma_{u_2}(\tau) \equiv y_0+(u_2-y_0)\tau
+2(u_2-2y_1+y_0)\tau(\tau-1) = \tilde
\gamma(\tau)+4(u_1-y_1)\tau(\tau-1),
$$
from the first equation in \eqref{orthbvm} and \eqref{Phi} we get
$$\begin{array}{rl}
\Phi(u_2) = & \displaystyle y_0+2hJ\int_0^1\nabla
H(\gamma_{u_2}(\tau)) \,\d \tau  = y_0+2hJ\int_0^1\nabla
H(\tilde \gamma(\tau)) \,\d \tau  \\[.4cm]
& \displaystyle  + 8hJ \int_0^1 \nabla^2 H(\tilde
\gamma(\tau))\tau(\tau-1) \,\d \tau \cdot (u_1-y_1) +
O(||u_1-y_1||^2) \\[.4cm]
=& \displaystyle u_2+O(h^5).
\end{array}
$$
If $h$ is small enough, $u_2$ will be inside the ball $B(y_0,\rho)$
defined in Lemma \ref{lem1}. The Lipschitz condition yields (see
\eqref{Lcond})
$$
||\hat z -u_2|| = ||\Phi(\hat z) -\Phi(u_2) +O(h^5)|| \le
\frac{h}{2}M_\rho ||\hat z-u_2|| +O(h^5),
$$
and hence $||\hat z -u_2|| =  O(h^5)||$.
\end{proof}

The above result states that \eqref{M} defines a method of order $4$
which is a simplified (non corrected) version of our conservative
method defined at \eqref{nonlin_sys}. In Section \ref{test_sec} the
behavior of these two methods will be compared on a set of test
problems. We now state the analogous results for system
\eqref{nonlin_sys}.

\begin{theorem}
\label{th1} Under the assumption ($A_1$), for $h$ small enough,
equation \eqref{nonlin_sys} admits a unique solution $z^\ast$
satisfying $y(t+2h)-z^\ast=O(h^5)$.
\end{theorem}
\begin{proof}
Consider the solution $\hat z$ of system \eqref{M}. We have (see
\eqref{tgamma})
$$
\gamma_{\hat z}(\tau)-\tilde \gamma(\tau)=(\hat z -u_2)\tau(2\tau-1)
+4(u_1-y_1)\tau(\tau-1)=O(h^5),
$$
and
$$
\hat z-2y_1+y_0=u_2-2u_1+y_0 +O(h^5).
$$
Hence, by virtue of \eqref{orthbvm},
$$
r(\hat z) = -2\left[(u_2-2u_1+y_0) +O(h^5)\right]^T \left[\int_0^1
(2\tau-1)\nabla H(\tilde \gamma(\tau)) \, \d \tau
+O(h^5)\right]=O(h^5).
$$
Since $a(\hat z)$ is bounded with respect to $h$, it follows that,
in a neighborhood of $\hat z$, system \eqref{nonlin_sys} may  be
regarded as a perturbation of system \eqref{M}, the perturbation
term being $R(z,h)\equiv \frac{r(z)}{||a(z)||_2^2}a(z)$.

Consider the ball $B(\hat z, R(\hat z, h))$: since $\hat
z=y_0+O(h)$, and $R(\hat z, h)=O(h^5)$, this ball is contained in
$B(y_0, \rho)$ defined in Lemma \ref{lem1} and the perturbed
function $\Phi(z)+R(z,h)$ is a contraction therein, provided $h$ is
small enough. Evaluating the right-hand side of \eqref{nonlin_sys}
at $z=\hat z$ we get
$$
y_0+2h J a(\hat z) + R(\hat z, h) = \hat z + R(\hat z, h),
$$
which means that property (b) listed in the proof of Lemma
\ref{lem1}, with $\hat z$ in place of $y_0$, holds true for the
perturbed function $y_0+2h J a(z) + R(z, h)$, and the contraction
mapping theorem may be again exploited to deduce the assertion.
\end{proof}

\section{Discretization}
\label{discr_sec} As was stressed in Section \ref{def_methods},
formula \eqref{nonlin_sys} is not operative unless a technique to
solve the two integrals is taken into account. The most obvious
choice is to compute the integrals by means of a suitable quadrature
formula which may be assumed  exact in the case where the
Hamiltonian function is a polynomial, and to provide an
approximation to within machine precision in all other cases.

Hereafter we assume that $H(q,p)$ is a polynomial in $q$ and $p$ of
degree $\nu$. Since $\gamma(\tau)$ has degree two, it follows that
the integrand functions appearing in the definitions of $a(z)$ and
$r(z)$ at \eqref{nonlin_sys} and \eqref{residual} have degree
$2\nu-2$ and $2\nu-1$ respectively and can be solved by any
quadrature formula with abscissae $c_1<c_2<\cdots<c_k$ in $[0,1]$
and weights $b_1,\dots,b_k$, having degree of precision $d \ge
2\nu-1$. In place of \eqref{nonlin_sys} we now consider the
equivalent form suitable for implementation
\begin{equation}
\label{twostep} \displaystyle y_2=y_0+ 2 h J \sum_{i=1}^kb_i\nabla
H(\gamma(c_i)) + G(y_0,y_1,y_2),
\end{equation}
where
$$
G(y_0,y_1,y_2) = \frac{-2(y_2-2y_1+y_0)^T \sum_{i=1}^k b_i (2c_i-1)
\nabla H(\gamma(c_i))}{\| \sum_{i=1}^kb_i\nabla H(\gamma(c_i))
\|_2^2} \, \sum_{i=1}^kb_i\nabla H(\gamma(c_i)).
$$
Notice that from \eqref{sigma} we get
\begin{equation}
\label{lin-comb}
\gamma(c_i)=(1-3c_i+2c_i^2)y_0+4c_i(1-c_i)y_1+c_i(2c_i-1)y_2,
\end{equation}
that is, $\gamma(c_i)$ is a linear combination, actually a weighted
average, of the approximations $y_0$, $y_1$ and $y_2$. Therefore,
since $G(y_0,y_1,y_2) = O(h^5)$ (see Lemma~\ref{lem2} and
Theorem~\ref{th1}), we may look at this term as a nonlinear
correction of the generalized linear multistep method
\begin{equation}
\label{twostep-lin} \displaystyle y_2=y_0+ 2 h J
\sum_{i=1}^kb_i\nabla H(\gamma(c_i)).
\end{equation}
\begin{exa}
If $H(q,p)$ is quadratic, we can choose $k=3$, $c_1=0$,
$c_2=\frac{1}{2}$, $c_3=1$, $b_1=b_3=\frac{1}{6}$ and
$b_2=\frac{2}{3}$, that is we can use Simpson's quadrature formula
to compute the integrals in \eqref{nonlin_sys} and \eqref{residual}.
Since, in such a case, $\gamma(c_i)=y_{i-1}$, method
\eqref{twostep-lin} becomes
$$
\displaystyle y_2=y_0+ \frac{h}{3} J \left(\nabla H(y_0)+4 \nabla
H(y_1)+\nabla H(y_2) \right),
$$
that is, the standard Milne-Simpson's method.
\end{exa}

In all other cases $\gamma(c_i)$ will differ in general from $y_j$,
$j=1,2,3$ and may be regarded as an off-point entry in formula
\eqref{twostep-lin}. In the sequel we will denote the method defined
at \eqref{twostep} by $M_k$ and its linear part, defined at
\eqref{twostep-lin}, by $M'_k$. Of course, the choice of the
abscissae distribution influences the energy preserving properties
of the method $M_k$, as is indicated in Table
\ref{nodes-distribution-table}.
\begin{table}[htb]
\begin{center}
{\renewcommand{\arraystretch}{1.5}
\begin{tabular}{|cccc|}
\hline {Abscissae distribution:} & uniform & Lobatto & Gauss \\
\hline { Energy preserving when:} &$\deg H \le \lceil \frac{k}{2}
\rceil$ & $ \deg H \le k-1$ & $ \deg H \le k$ \\ \hline
\end{tabular}}
\end{center}
\caption{Energy preserving properties of method $M_k$ for some
well-known distributions of the nodes $\{c_i\}$.}
\label{nodes-distribution-table}
\end{table}


\section{Numerical tests}
\label{test_sec} Hereafter we implement the order four method $M_k$
on  a few Hamiltonian problems to show that the numerical results
are consistent with the theory presented in Section
\ref{analysis_sec}. In particular, in the first two problems the
Hamiltonian function is a polynomial of degree three and six
respectively, while the last numerical test reports the behavior of
the method on a non-polynomial problem.

Each step of the integration procedure requires the solution  of a
nonlinear system, in the unknown $y_2$, represented by
\eqref{twostep} for the method $M_k$ and \eqref{twostep-lin} for the
method $M'_k$. The easiest way (although not the most efficient one)
to find out a solution is by means of fixed point iteration that, in
the case of the method $M_k$, reads
\begin{equation}
\label{iteration} z_{s+1}=y_0+ 2 h J \sum_{i=1}^kb_i\nabla
H(\gamma_{z_s}(c_i)) + G(y_0,y_1,z_s),\qquad s=1,2,\dots,
\end{equation}
where $\gamma_z$ is defined at \eqref{gammaz} and $z_0$ is an
initial approximation of $y_2$ which is then refined by setting
$y_2=z_{\bar s}$ with $z_{\bar s}\simeq \lim_{s\rightarrow \infty}
z_s$. From Theorem \ref{th1} and the preceding lemmas we deduce that
such a limit always exists provided that $h$ is small enough. The
value of $z_0$ could be retrieved via an extrapolation based on the
previous computed points or by considering the method $M'_k$ as a
predictor for $M_k$.

We will consider a Lobatto distribution with an odd number $k$ of
abscissae $\{c_i\}$. In fact, if $k$ is odd, since
$y_0=\gamma(0)=\gamma(c_1)$ and
$y_1=\gamma(\frac{1}{2})=\gamma(c_{\lceil \frac{k}{2}\rceil})$, we
save two function evaluations during the iteration
\eqref{iteration}.

\subsection{Test problem 1}
The Hamiltonian function
\begin{equation}
\label{cubic_pendulum}
H(q,p)=\frac{1}{2}p^2+\frac{1}{2}q^2-\frac{1}{6}q^3
\end{equation}
defines the cubic pendulum equation. We can solve it by using five
Lobatto nodes to discretize the integrals in \eqref{nonlin_sys},
thus getting the method $M_5$. The corresponding numerical solution,
denoted by $(q_n,p_n)$, is plotted in Figure
\ref{cubic_pendulum_fig1}. For comparison purposes we also compute
the numerical solution $(q'_n, p'_n)$ provided by the fourth order
method, say $M'_5$, obtained by neglecting in \eqref{nonlin_sys} the
correction term, that is  by posing $r(z)\equiv 0$. Figure
\ref{cubic_pendulum_fig2} clearly shows the energy conservation
property, while Table \ref{tab1} summarizes the convergence
properties of the two methods.

\begin{figure}[h]
\begin{center}
\includegraphics[width=6.7cm, height=5cm]{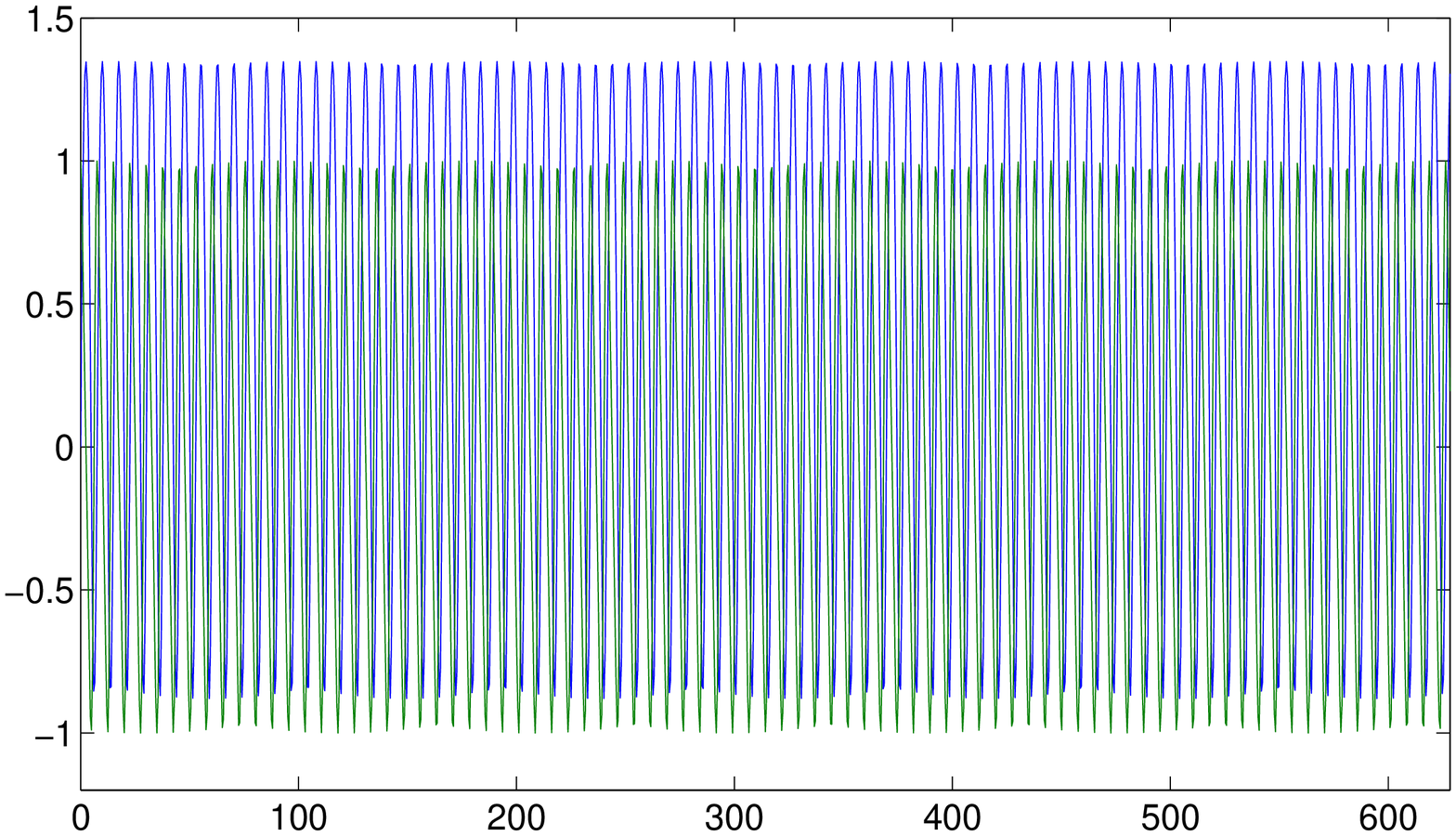}
\hspace*{-.6cm}
\includegraphics[width=6.7cm, height=5cm]{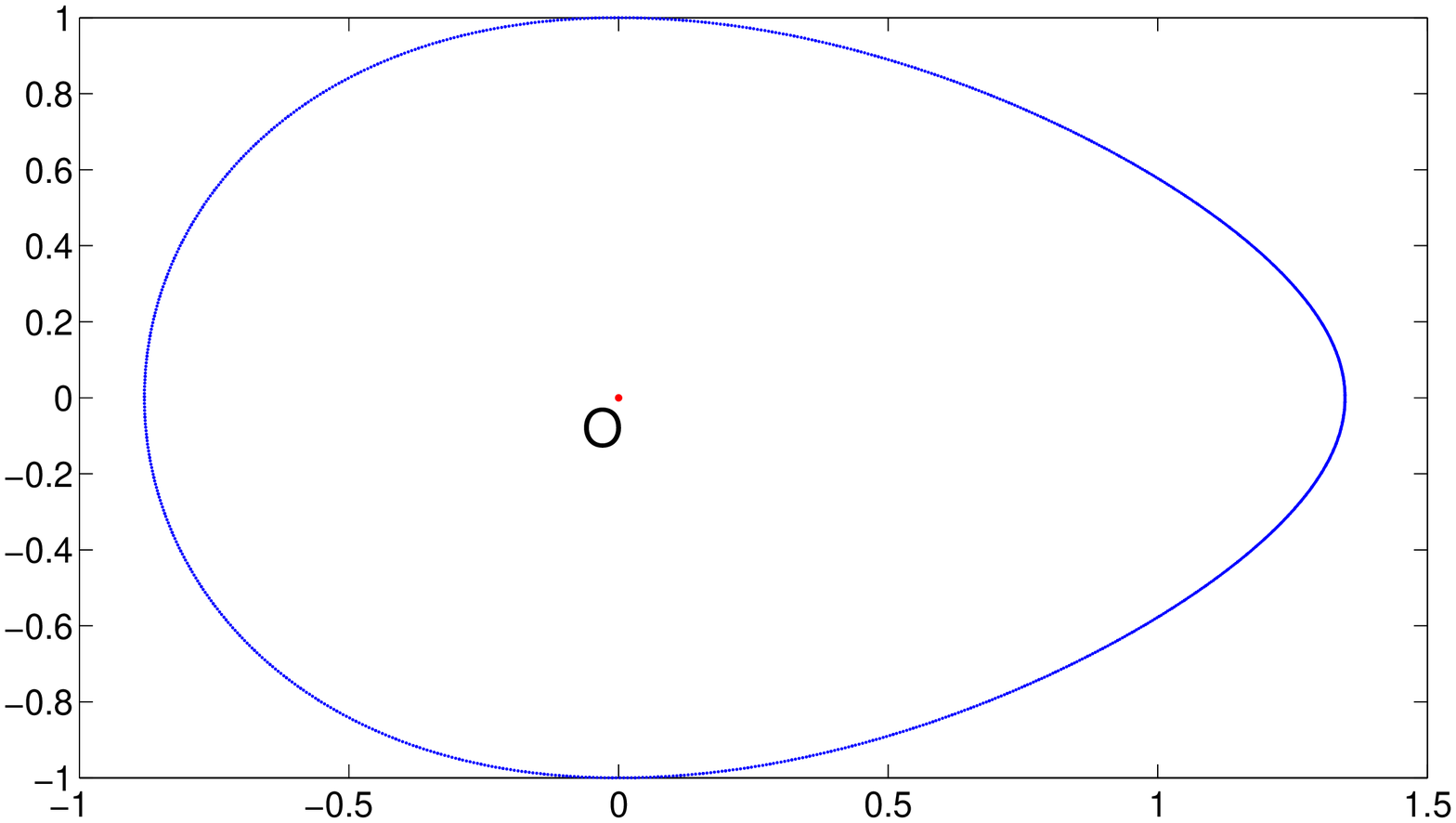}
\end{center}
\vspace*{-.5cm} \caption{Numerical solution $(q_n, p_n)$ versus time
$t_n$ (left picture) and on the phase plane (right picture).
Parameters: initial condition $y_0=[0,\,1]$; stepsize $h=0.5$;
integration interval $[0, 200 \pi]$.} \label{cubic_pendulum_fig1}
\end{figure}
\begin{figure}[h]
\begin{center}
\includegraphics[width=10cm, height=5cm]{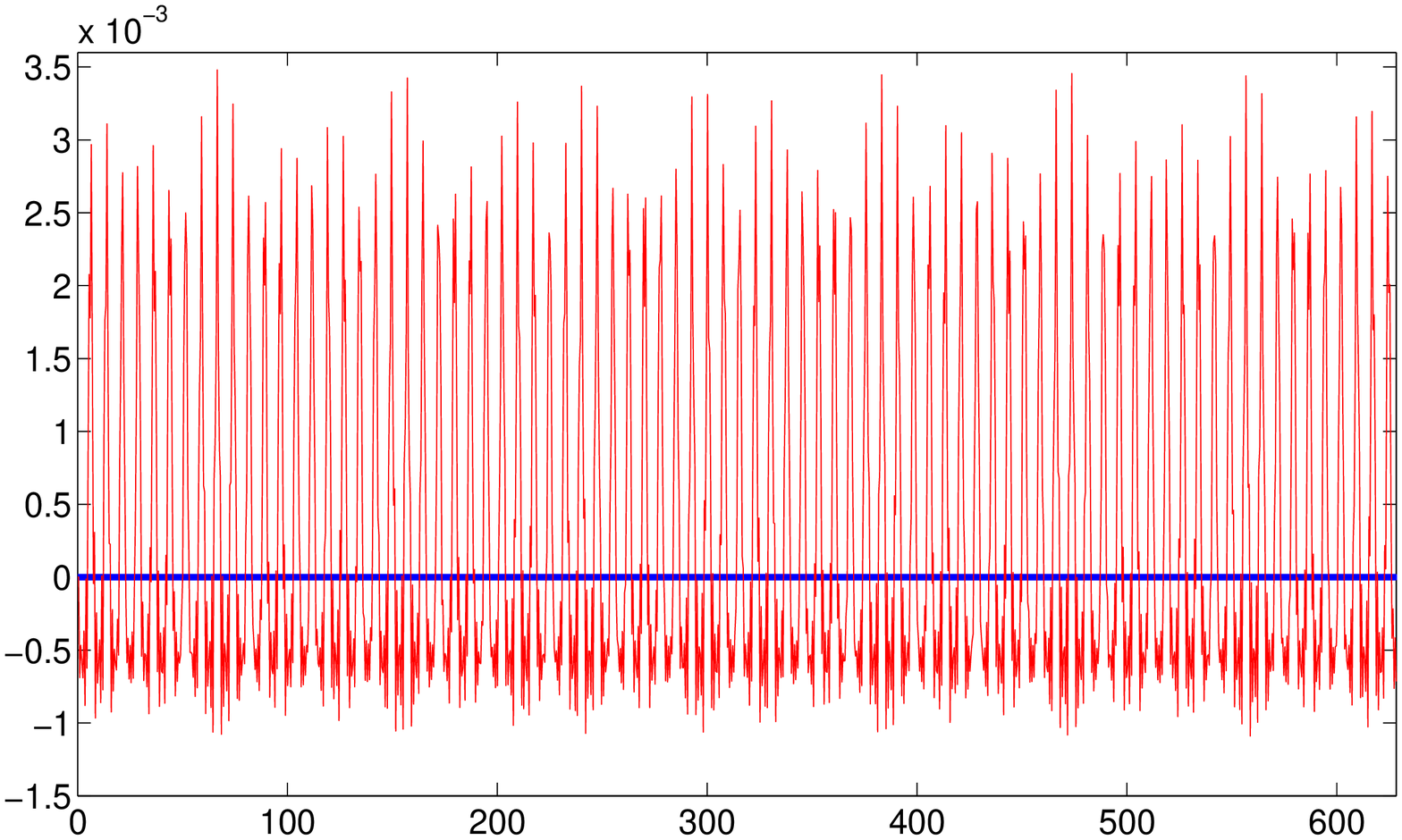}
\end{center}
\vspace*{-.5cm} \caption{Hamiltonian function evaluated along the
numerical solution $(p_n,q_n)$ (horizontal line) and along the
numerical solution $(p'_n, q'_n)$ (irregularly oscillating line).}
\label{cubic_pendulum_fig2}
\end{figure}

\begin{table}[h]
\label{tab1}{\footnotesize
\begin{center}
\begin{tabular}{c|c|c|c|c|c|c|}
\cline{2-7} & \multicolumn{3}{c|}{method $M_5$} & \multicolumn{3}{c|}{method $M'_5$} \\
\hline
\multicolumn{1}{|c|}{$h$}  & error  & order  & {\tiny $\max |H(y_n)-H(y_0)|$} & error  & order & {\tiny $\max |H(y'_n)-H(y_0)|$} \\
\hline
\multicolumn{1}{|c|}{$1$}      &   $3.1\cdot 10^{-2}$  &           & $2.5\cdot 10^{-15}$ & $1.1\cdot 10^{-1}$&         & $1.1008\cdot 10^{-1}$ \\[.1cm]
\multicolumn{1}{|c|}{$2^{-1}$} &   $3.8\cdot 10^{-4}$  & $6.373$   & $1.9\cdot 10^{-15}$ & $3.1\cdot 10^{-3}$& $5.183$ & $2.9680\cdot 10^{-3}$ \\[.1cm]
\multicolumn{1}{|c|}{$2^{-2}$} &   $2.6\cdot 10^{-5}$  & $3.866$   & $1.5\cdot 10^{-15}$ & $2.5\cdot 10^{-4}$& $3.655$ & $1.5755\cdot 10^{-4}$ \\[.1cm]
\multicolumn{1}{|c|}{$2^{-3}$} &   $1.6\cdot 10^{-6}$  & $4.059$   & $8.8\cdot 10^{-16}$ & $1.8\cdot 10^{-5}$& $3.811$ & $8.5163\cdot 10^{-6}$ \\[.1cm]
\multicolumn{1}{|c|}{$2^{-4}$} &   $9.5\cdot 10^{-8}$  & $4.032$   & $9.9\cdot 10^{-16}$ & $1.2\cdot 10^{-6}$& $3.905$ & $4.8883\cdot 10^{-7}$ \\[.1cm]
\multicolumn{1}{|c|}{$2^{-5}$} &   $5.9\cdot 10^{-9}$  & $4.017$   & $1.1\cdot 10^{-15}$ & $7.6\cdot 10^{-8}$& $3.952$ & $2.9131\cdot 10^{-8}$ \\[.1cm]
\multicolumn{1}{|c|}{$2^{-6}$} &   $3.6\cdot 10^{-10}$ & $4.008$   & $1.1\cdot 10^{-15}$ & $4.9\cdot 10^{-9}$& $3.976$ & $1.7771\cdot 10^{-9}$ \\[.1cm]
\multicolumn{1}{|c|}{$2^{-7}$} &   $2.3\cdot 10^{-11}$ & $4.004$   & $2.3\cdot 10^{-15}$ & $3.1\cdot 10^{-10}$&$3.988$ & $1.0968\cdot 10^{-10}$ \\[.1cm]
\multicolumn{1}{|c|}{$2^{-8}$} &   $1.4\cdot 10^{-12}$ & $4.006$   & $2.4\cdot 10^{-15}$ & $1.9\cdot 10^{-11}$&$3.994$ & $6.8121\cdot 10^{-12}$ \\
 \hline
\end{tabular}\end{center}}
\vspace*{.1cm} \caption{Methods $M_5$ (with correction term) and
$M'_5$ (without correction term) are implemented on the cubic
pendulum equation \eqref{cubic_pendulum} on the time interval $[0,
10]$ for several values of the stepsize $h$. The order of
convergence is numerically evaluated by means of the formula $\log_2
\frac{\mbox{\rm error}(\frac{h}{2})}{\mbox{\rm error}(h)}$. As was
expected, the maximum displacement of the numerical Hamiltonian
$H(y_n)$ from the theoretical value $H(y_0)$ is close to the machine
precision for the method $M_5$, independently of the stepsize $h$
used.}
\end{table}

\subsection{Test problem 2}
The Hamiltonian function
\begin{equation}
\label{fhp} H(p,q)= \frac{1}{3}p^3- \frac{1}{2}p
+\frac{1}{30}q^6+\frac{1}{4}q^4-\frac{1}{3}q^3+\frac{1}{6}
\end{equation}
has been proposed in \cite{FaHaPh} to show that symmetric methods
may suffer from the energy drift phenomenon even when applied to
reversible systems, that is when $H(-p,q)=H(p,q)$.\footnote{In fact,
the authors show that the system deriving from \eqref{fhp}  is
equivalent to a reversible system (see also \cite{BrTr0, McPe} for a
discussion on the integration of reversible Hamiltonian systems by
symmetric methods).} For our experiment, we will use $y_0=[0.2,\,
0.5]$ as initial condition.

Since  $\deg(H(q,p))=6$, we need a Lobatto quadrature based on at
least seven nodes to assure that the integrals in \eqref{nonlin_sys}
are computed exactly. Therefore we solve \eqref{fhp} by method
$M_7$. For comparison purposes, it is also interesting to  show the
dynamics of the symmetric non-conservative method $M'_7$. Figure
\ref{fhp_fig1} displays the results obtained by the two methods
implemented with stepsize $h=\frac{1}{10}$ over the interval
$[0,\,10^3]$. In particular, the numerical trajectories  generated
by method $M'_7$ and $M_7$,  are reported in the left-top and
left-bottom pictures respectively, while the right picture  reports
the corresponding error in the Hamiltonian function evaluated along
the two numerical solutions, namely $|H(y_n)-H(y_0)|$.

Evidently, the numerical solution produced by $M'_7$ rapidly departs
from the level curve $H(q,p)=H(q_0,p_0)$ but it remain eventually
bounded and the points $(q_n,p_n)$ seem to densely fill a bounded
region of the phase plane.

On the contrary, since the degree of freedom of the present problem
is one, the points $(q_n,p_n)$ produced by $M_7$ lie on the very
same continuous trajectory covered by $y(t)$: this is also confirmed
by looking at the bottom graph in the right picture.

Table \ref{tab2} shows the  behavior of method $M_7$ applied to
problem \eqref{fhp} as the stepsizes $h$ goes to zero. Notice the
$O(h^5)$ rate of convergence to zero for the residual function
$r(z)$   in \eqref{residual}.

\begin{figure}[h]
\begin{center}
\includegraphics[width=6.7cm, height=5cm]{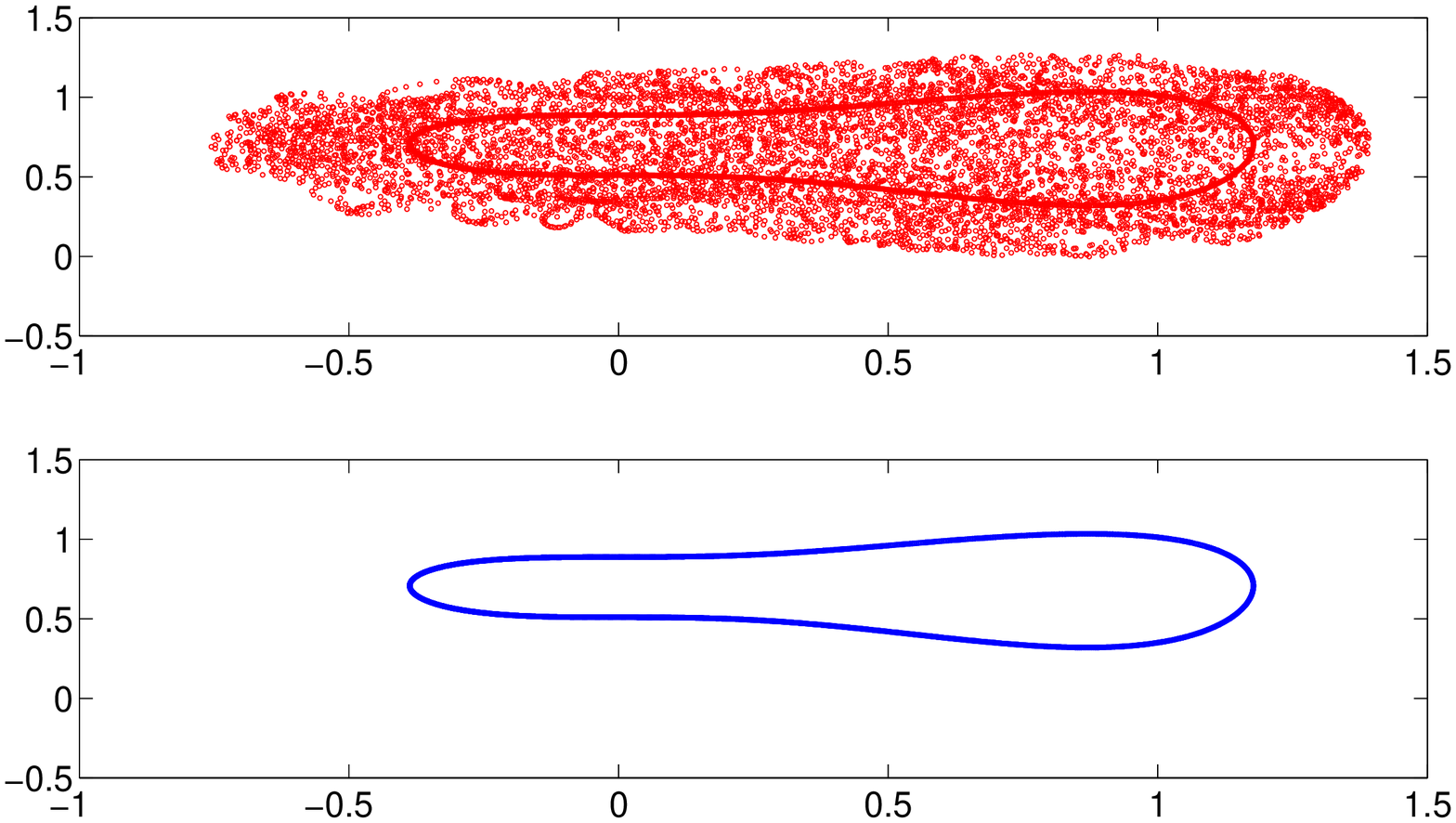}
\hspace*{-.6cm}
\includegraphics[width=6.7cm, height=5cm]{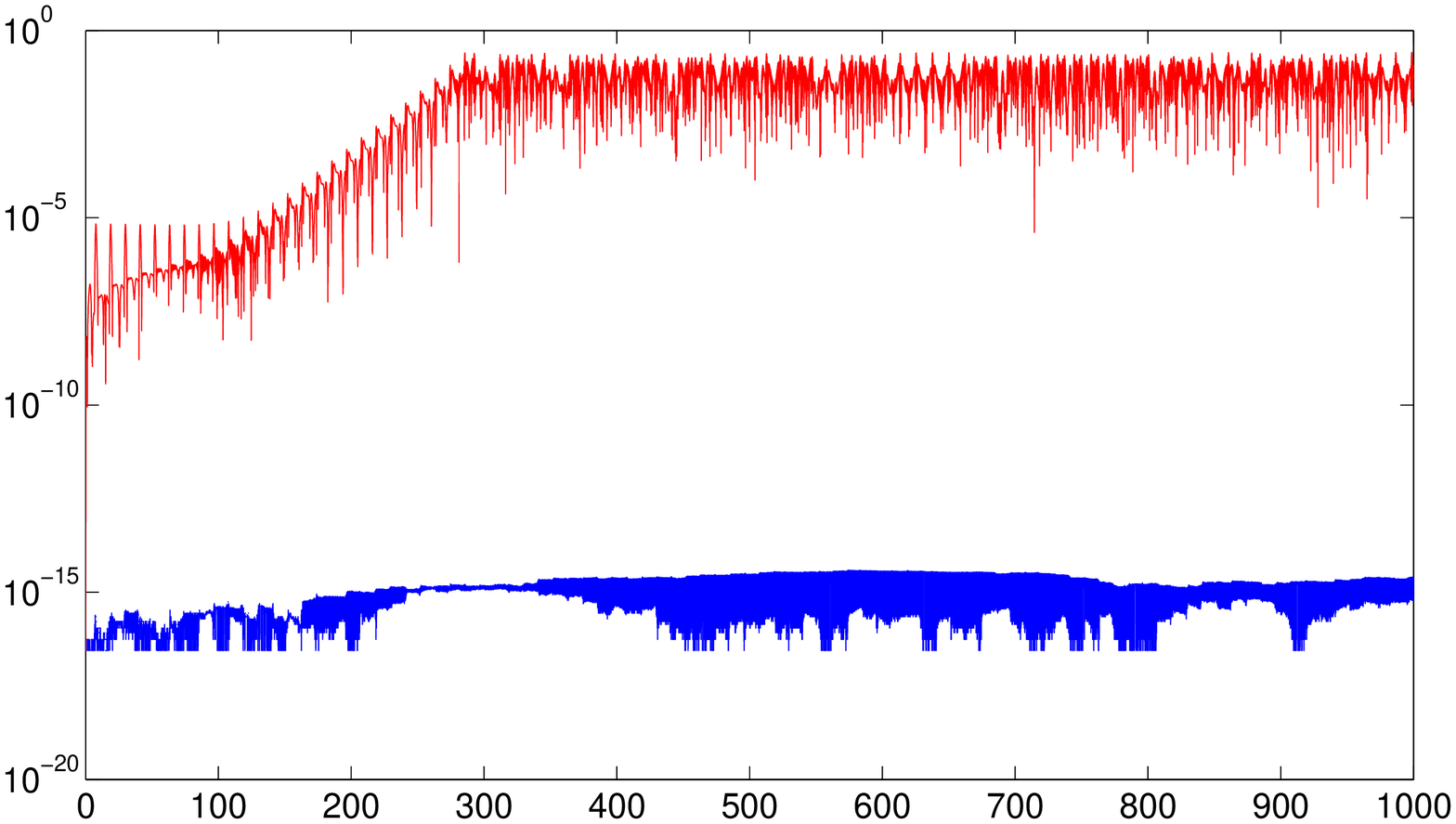}
\end{center}
\vspace*{-.5cm} \caption{Left pictures: numerical solutions in the
phase plane computed by method $M'_7$ (top picture) and $M_7$
(bottom picture). Right picture: error in the numerical Hamiltonian
function $|H(y_n)-H(y_0)|$ produced by the two methods. Parameters:
initial condition $y_0=[0.2,\,0.5]$; stepsize $h=0.1$; integration
interval $[0, 1000]$.} \label{fhp_fig1}
\end{figure}

\begin{table}[h]
\label{tab2}{\footnotesize
\begin{center}
\begin{tabular}{c|c|c|c|c|c|}
\cline{2-6} & \multicolumn{5}{c|}{method $M_7$}  \\
\hline
\multicolumn{1}{|c|}{$h$}  & error  & order  & { $|H(y_N)-H(y_0)|$} & residual $r(y_N)$  & order of $r(y_N)$  \\
\hline
\multicolumn{1}{|c|}{$2^{-1}$} &   $4.47\cdot 10^{-2}$  & $     $   & $1.6\cdot 10^{-16}$ & $-1.21\cdot 10^{-03}$& $     $ \\[.1cm]
\multicolumn{1}{|c|}{$2^{-2}$} &   $7.38\cdot 10^{-4}$  & $5.920$   & $4.4\cdot 10^{-16}$ & $-3.23\cdot 10^{-06}$& $8.559$ \\[.1cm]
\multicolumn{1}{|c|}{$2^{-3}$} &   $3.90\cdot 10^{-5}$  & $4.243$   & $5.8\cdot 10^{-16}$ & $-2.15\cdot 10^{-08}$& $7.225$ \\[.1cm]
\multicolumn{1}{|c|}{$2^{-4}$} &   $2.39\cdot 10^{-6}$  & $4.027$   & $2.4\cdot 10^{-16}$ & $-6.61\cdot 10^{-10}$& $5.029$ \\[.1cm]
\multicolumn{1}{|c|}{$2^{-5}$} &   $1.49\cdot 10^{-7}$  & $4.007$   & $2.5\cdot 10^{-15}$ & $-2.03\cdot 10^{-11}$& $5.021$ \\[.1cm]
\multicolumn{1}{|c|}{$2^{-6}$} &   $9.27\cdot 10^{-9}$  & $4.002$   & $3.2\cdot 10^{-15}$ & $-6.27\cdot 10^{-13}$& $5.018$ \\[.1cm]
\multicolumn{1}{|c|}{$2^{-7}$} &   $5.77\cdot 10^{-10}$ & $4.006$   & $5.5\cdot 10^{-16}$ & $-2.00\cdot 10^{-14}$& $4.972$ \\[.1cm]
\multicolumn{1}{|c|}{$2^{-8}$} &   $3.16\cdot 10^{-11}$ & $4.188$   & $5.4\cdot 10^{-15}$ & $-5.36\cdot 10^{-16}$& $5.219$ \\
 \hline
\end{tabular}\end{center}}
\vspace*{.1cm} \caption{Performance of method $M_7$ applied to
problem \eqref{fhp}, with initial condition $y_0=[0.2,\,0.5]$, on
the time interval $[0, 250]$ for several values of the stepsize $h$,
as specified in the first column. The second and third columns
report the relative error in the last computed point $y_N$, $N=T/h$
and the corresponding order of convergence. Since the integrals
appearing in \eqref{nonlin_sys} are precisely computed by the
Lobatto quadrature formula with seven nodes, the error in the
numerical Hamiltonian $H(y_N)$ is zero up to machine precision. The
last two columns list the residual $r(y_N)$ defined in
\eqref{residual} and its order of convergence to zero.}
\end{table}

\subsection{Test problem 3}
We finally consider the non-polynomial Hamiltonian function
\begin{equation}
\label{kepler} H(q_1,q_2,p_1,p_2) =
\frac{1}2(p_1^2+p_2^2)-\frac{1}{\sqrt{q_1^2+q_2^2}}
\end{equation}
that defines the well known Kepler problem,  namely  the motion of
two masses under the action of their mutual gravitational
attraction. Taking as initial condition
\begin{equation}
\label{kepler0}
(q_1(0),q_2(0),p_1(0),p_2(0))=\pmatrix{cccc}1-e,&0,&0,
&\sqrt{\frac{1+e}{1-e}}\endpmatrix^T
\end{equation}
yields an elliptic periodic orbit of period $2\pi$ and eccentricity
$e\in[0,1)$. We have chosen $e=0.6$. Though the vector field fails
to be a polynomial in $q_1$ and $q_2$, we can plan to use a
sufficiently large number of quadrature nodes to discretize the
integrals in \eqref{nonlin_sys} so that the corresponding accuracy
is within the machine precision. Under this assumption, and taking
aside the effect of the floating point arithmetic, the computer will
make no difference between the conservative formulae
\eqref{nonlin_sys} and their discrete counterparts.

The left picture in Figure \ref{kepler_fig1} explains the above
argument. It reports the error $|H(y_n)-H(y_0)|$ in the Hamiltonian
function for various choices of the number of Lobatto nodes, and
precisely $k=3,\,5,\,7,\,9$. We see that the error decreases quickly
as the number of nodes is incremented and for $k=9$ it is within the
epsilon machine.\footnote{All tests were performed in Matlab using
double precision arithmetic.}

The use of finite arithmetic may sometimes cause a mild numerical
drift of the energy over long times, like the one shown in the upper
line in the right picture of Figure \ref{kepler_fig1}. This is due
to the fact that on a computer the numerical solution satisfy the
conservation relation $H(y_n)=H(y_0)$ up to machine precision times
the conditioning number of the nonlinear system that is to be solved
at each step.

To prevent the accumulation of roundoff errors we may apply a simple
and costless \textit{correction} technique  on the approximation
$y_n$ which consists in a single step of a gradient descent method
(see also \cite{BIS}). More precisely, the corrected solution
$y^*_n$ is defined by
\begin{equation}
\label{descent} y^*_n = y_n-\alpha \frac{\nabla H(y_n)}{||\nabla
H(y_n)||_2}, \qquad \mbox{with} ~
\alpha=\frac{H(y_n)-H(y_0)}{||\nabla H(y_n)||_2},
\end{equation}
which stems from choosing as $\alpha$ the value that minimizes the
linear part of the function $F(\alpha) = H(y_n-\alpha \frac{\nabla
H(y_n)}{||\nabla H(y_n)||_2})-H(y_0)$. The bottom line in the right
picture of Figure \ref{kepler_fig1} shows the energy conservation
property of the corrected solution.

\begin{figure}[h]
\begin{center}
\includegraphics[width=6.7cm, height=5cm]{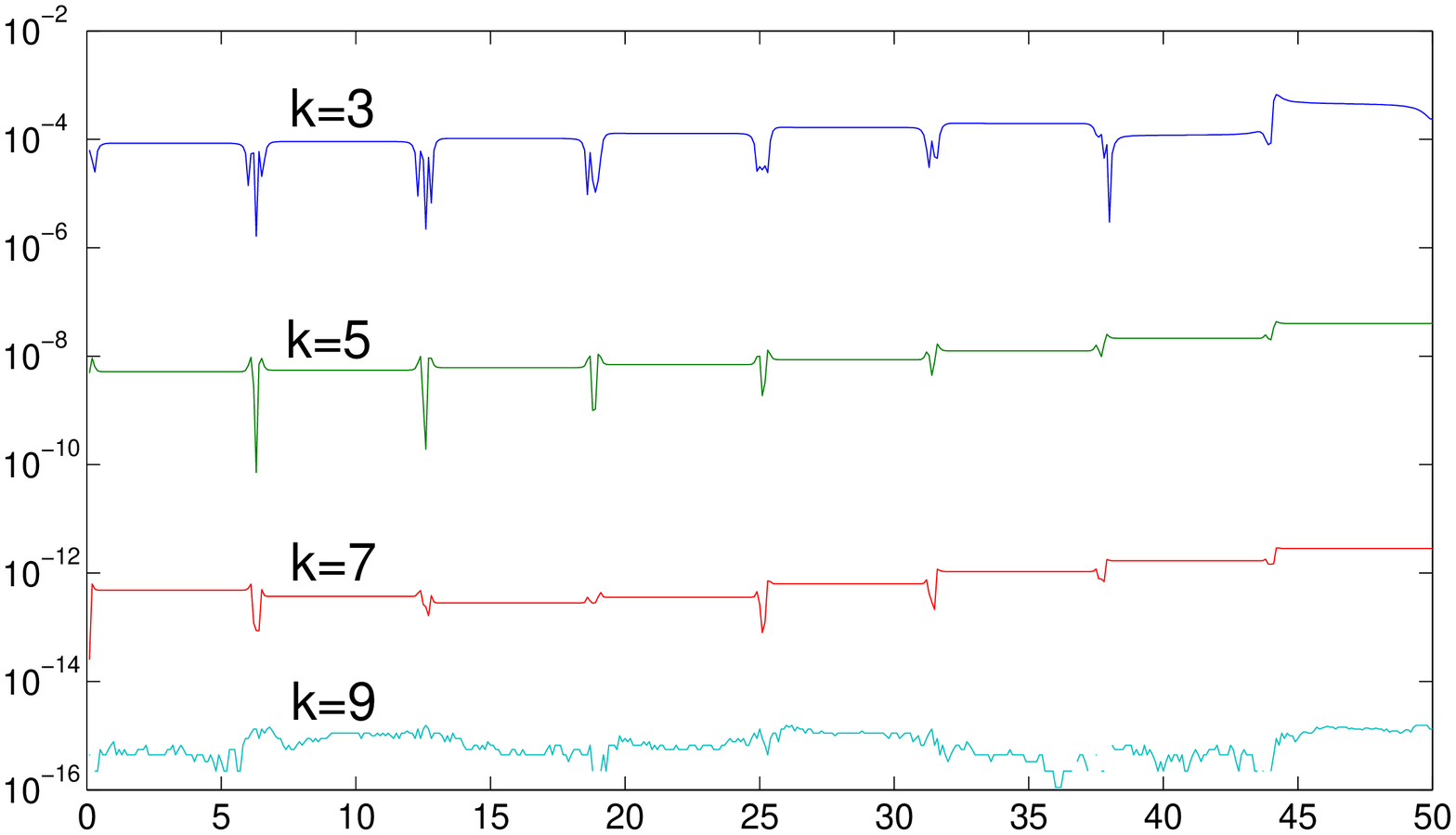}
\hspace*{-.6cm}
\includegraphics[width=6.7cm, height=5cm]{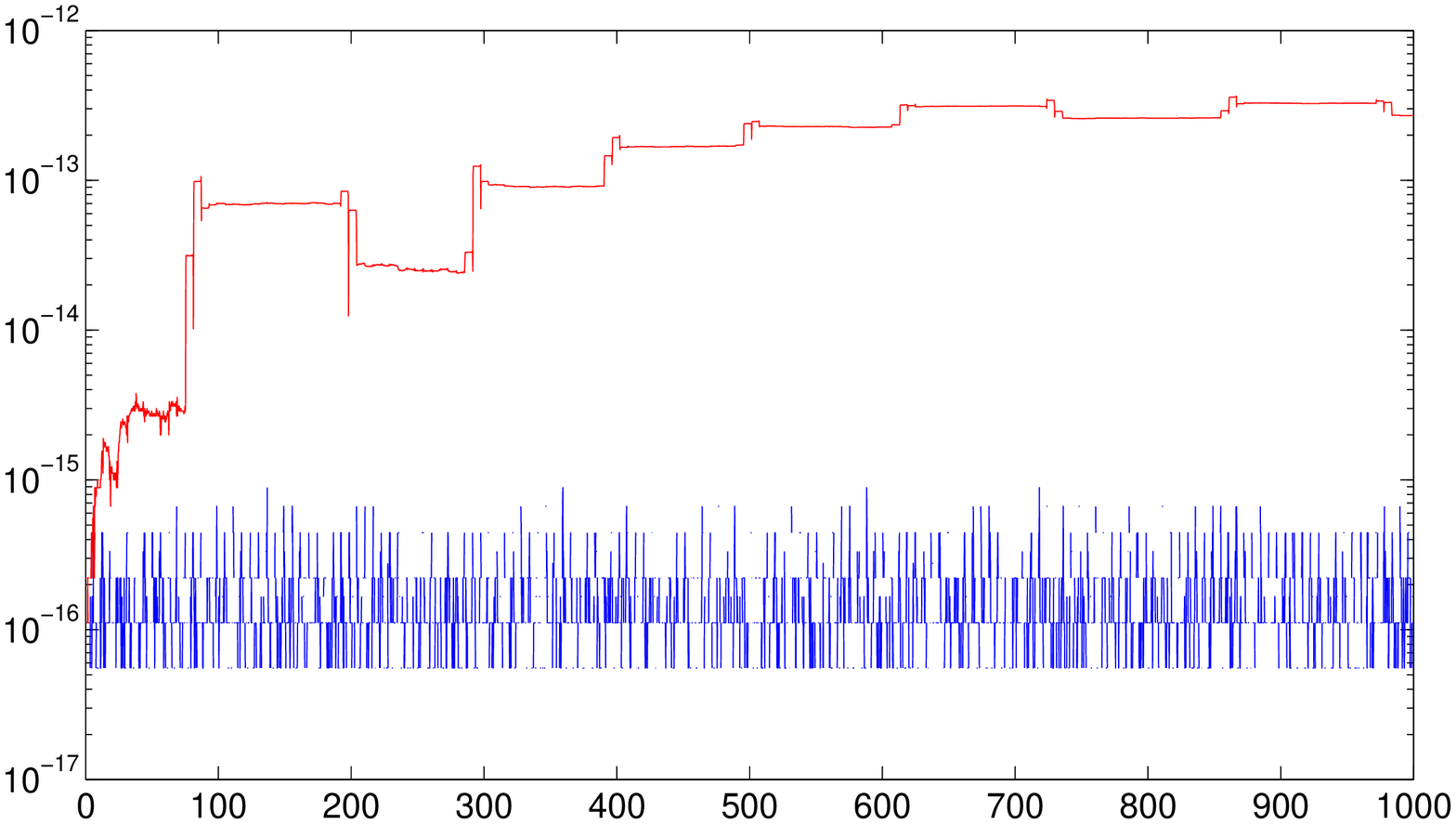}
\end{center}
\vspace*{-.5cm} \caption{Left picture. Error in the numerical
Hamiltonian function $|H(y_n)-H(y_0)|$ produced by methods $M_k$,
with $k=3,\,5,\,7,\,9$. Parameters: stepsize $h=0.05$, integration
interval $[0, 50]$. Right picture.  Roundoff errors may cause a
drift of the numerical Hamiltonian function (upper line) which can
be easily taken under control by coupling the method with a costless
correction procedure like the one described at \eqref{descent}.}
\label{kepler_fig1}
\end{figure}

\section{Conclusions}
We have derived a family of mono-implicit methods of order four with
energy-preserving properties. Each element in the family originates
from a limit formula and is defined by discretizing the integral
therein by means of a suitable quadrature scheme. This process
assures an exact energy conservation in the case where the
Hamiltonian function is a polynomial, or a conservation to within
machine precision in all other cases, as is also illustrated in the
numerical tests. Interestingly, each method may be conceived as a
$O(h^5)$ perturbation of a two-step linear method.

\end{document}